\documentclass[11pt]{amsart}
\usepackage[colorlinks,citecolor=blue,urlcolor=black,linkcolor=black]{hyperref}
\usepackage{amssymb, latexsym}
\usepackage{enumerate}
\usepackage[headings]{fullpage}
\usepackage{cmap}

\newtheorem{thm}{Theorem}[section]
\newtheorem*{PROPA}{Proposition A}
\newtheorem*{PROPB}{Proposition B}
\newtheorem*{PROPC}{Proposition C}
\newtheorem*{THM}{Main Theorem}
\newtheorem*{CORA}{Corollary A}
\newtheorem*{CORB}{Corollary B}
\newtheorem{lemma}[thm]{Lemma}
\newtheorem{cor}[thm]{Corollary}
\newtheorem{claim}{Claim}[thm]
\newtheorem{prop}[thm]{Proposition}

\theoremstyle{definition}
\newtheorem{defn}[thm]{Definition}

\theoremstyle{remark}
\newtheorem*{remarks}{Remarks}

\newcommand\name[1]{\dot{#1}}
\newcommand\forces{\Vdash}
\newcommand\s{\subseteq}
\newcommand\sq{\sqsubseteq}
\newcommand{\sqleft}[1]{\mathrel{_{#1}{\sqsubseteq}}}
\newcommand{\sqx}{\sqleft{\chi}}
\newcommand\br{\blacktriangleright}

\renewcommand{\restriction}{\mathbin\upharpoonright}    
\DeclareMathOperator{\suc}{succ}
\DeclareMathOperator{\height}{ht}   
\DeclareMathOperator{\cf}{cf}
\DeclareMathOperator{\dom}{dom}
\DeclareMathOperator{\rng}{Im}
\DeclareMathOperator{\otp}{otp}
\DeclareMathOperator{\acc}{acc}
\DeclareMathOperator{\nacc}{nacc}
\DeclareMathOperator{\col}{Col}

\DeclareMathOperator{\p}{P}

\newcommand\axiomfont[1]{\textsf{\textup{#1}}}
\newcommand\zfc{\axiomfont{ZFC}}
\newcommand\gch{\axiomfont{GCH}}
\newcommand\ch{\textup{CH}}
\newcommand\ns{\textup{NS}}

\subjclass[2010]{Primary 03E05; Secondary 03E35, 05C05}

\begin{document}
\begin{abstract} An $\aleph_1$-Souslin tree is a complicated combinatorial object whose existence cannot be decided on the grounds of \zfc\ alone.
But 15 years after Tennenbaum and independently Jech devised notions of forcing for introducing such a tree,
Shelah proved that already the simplest forcing notion --- Cohen forcing --- adds an $\aleph_1$-Souslin tree.

In this paper, we identify a rather large class of notions of forcing that, assuming a $\gch$-type assumption, add a $\lambda^+$-Souslin tree.
This class includes Prikry, Magidor and Radin forcing.
\end{abstract}

\author{Ari Meir Brodsky}
\author{Assaf Rinot}
\address{Department of Mathematics, Bar-Ilan University, Ramat-Gan 5290002, Israel.}
\urladdr{\url{http://u.math.biu.ac.il/~brodska/}}
\urladdr{\url{http://www.assafrinot.com}}
\thanks{This research was supported by the Israel Science Foundation (grant $\#$1630/14).}

\title{More notions of forcing add a Souslin tree}

\maketitle

\section{Introduction}
The definition of a $\kappa$-Souslin tree may be found in Section~2.
Our starting point is a theorem of Jensen from his masterpiece \cite{MR309729}.
Let $\kappa$ denote a regular uncountable cardinal, and let $E$ denote a stationary subset of $\kappa$.
The principle from \cite[p.~287]{MR309729}, which we denote here by $\square(E)$, asserts the existence of a sequence $\langle C_\alpha\mid\alpha<\kappa\rangle$
such that for every limit ordinal $\alpha<\kappa$:
\begin{itemize}
\item $C_\alpha$ is a club in $\alpha$;
\item if $\bar\alpha$ is an accumulation point of $C_\alpha$, then $C_\alpha\cap\bar\alpha=C_{\bar\alpha}$ and $\bar\alpha\notin E$.\footnote{Note that $\square(E)$ for $E=\kappa=\aleph_1$ is a trivial consequence of $\zfc$,
and that $\square(E)$ for $\sup(E)=\kappa>\aleph_1$ implies that $E\neq\kappa$.
Therefore, our choice of notation does not conflict with the principle $\square(\kappa)$ from \cite[p.~267]{MR908147}.}
\end{itemize}
The principle from \cite[p.~293]{MR309729}, commonly denoted $\diamondsuit(E)$, asserts the existence of a sequence $\langle Z_\alpha\mid\alpha<\kappa\rangle$ such that for every subset $Z\s\kappa$,
there exist stationarily many $\alpha\in E$ such that $Z\cap\alpha=Z_\alpha$.
Jensen's theorem reads as follows:
\begin{thm}[Jensen, {\cite[Theorem~6.2]{MR309729}}] If $E$ is a stationary subset of a given regular uncountable cardinal $\kappa$,
and $\square(E)+\diamondsuit(E)$ holds, then there exists a $\kappa$-Souslin tree.
\end{thm}

The goal of this paper is to identify various forcing scenarios that will introduce $\kappa$-Souslin trees.
We do so by studying several combinatorial principles that can be used (together with $\diamondsuit(\kappa)$) to construct $\kappa$-Souslin trees,
and establishing that some forcing scenarios already introduce these.
The simplest among the combinatorial principles under consideration is the following:

\begin{defn}\label{xbox}
For any regular uncountable cardinal $\kappa$, $\boxtimes^*(\kappa)$ asserts
the existence of a sequence $\langle \mathcal C_\alpha\mid\alpha<\kappa\rangle$ such that:
\begin{enumerate}
\item For every limit ordinal $\alpha<\kappa$:
\begin{itemize}
\item$\mathcal C_\alpha$ is a nonempty collection of club subsets of $\alpha$, with $\left| \mathcal C_\alpha \right| < \kappa$;
\item if $C \in \mathcal C_\alpha$ and $\bar\alpha$ is an accumulation point of $C$, then $C\cap\bar\alpha\in\mathcal C_{\bar\alpha}$;
\end{itemize}
\item For every cofinal $A\s\kappa$, there exist stationarily many $\alpha<\kappa$ such that $\sup(\nacc(C)\cap A)=\alpha$ for all $C \in \mathcal C_\alpha$.\footnote{Here, $\nacc(C)$ stands for the set of non-accumulation points of $C$.
See the \emph{Notation} subsection below.}
\end{enumerate}
\end{defn}

An evident difference between the principles $\square(E)$ and $\boxtimes^*(\kappa)$ is that the former assigns only a single club to each level $\alpha$,
whereas the latter assigns many (like in Jensen's \emph{weak square} principle). A more substantial difference is that the principle $\square(E)$ implies that the stationary set $E$ is non-reflecting,
whereas the principle $\boxtimes^*(\kappa)$ is consistent with the statement that all stationary subsets of $\kappa$ reflect (by \cite{chris}, or by a combination of the main results of \cite{MR3498375} and \cite{rinot24}).

Nevertheless, $\boxtimes^*(\kappa)$ is a nontrivial principle.
For instance, one can use the function $\rho_2$ from the theory of \emph{walks on ordinals} \cite{MR908147} to show that $\boxtimes^*(\kappa)$ entails the existence of a $\kappa$-Aronszajn tree.
More importantly, we have the following:

\begin{PROPA} If $\kappa$ is a regular uncountable cardinal,
and $\boxtimes^*(\kappa)+\diamondsuit(\kappa)$ holds, then there exists a $\kappa$-Souslin tree.
\end{PROPA}

Now, there is an obvious way of introducing $\boxtimes^*(\kappa)$ by forcing. Conditions are sequences
$\langle \mathcal C_\alpha\mid\alpha\le\delta\rangle$ of successor length below $\kappa$,
such that for every limit ordinal $\alpha\le\delta$, the two bullets of Definition~\ref{xbox}(1) hold.
But is there another way?

The main result of this paper is the identification of a ceratin class of notions of forcing that (indirectly) introduce $\boxtimes^*(\kappa)$.
\begin{defn}\label{theclass}
For a regular uncountable cardinal $\lambda$, let $\mathbb C_\lambda$ denote the class of all notions of forcing $\mathbb P$ satisfying the two items:
\begin{enumerate}
\item $\mathbb P$ is \textit{$\lambda^+$-cc} and has size $\le2^\lambda$;
\item in $V^{\mathbb P}$, there exists a cofinal subset $\Lambda\s\lambda$ such that for every function $f\in{}^\lambda\lambda\cap V$, there exists some $\xi\in\Lambda$ with $f(\xi)<\min(\Lambda\setminus(\xi+1))$.
\end{enumerate}
\end{defn}

Clearly, assuming a $\gch$-type assumption, the various notions of forcing for adding a fast club to $\lambda$ (such as \cite[p.~97]{MR384542} and the minor variations \cite[p.~650]{MR716625}, \cite[p.~820]{rinot06}) belong to this class.
Also, for infinite regular cardinals $\theta<\lambda=\lambda^{<\theta}$, the L\'evy collapse $\col(\theta,\lambda)$ belongs to $\mathbb C_\lambda$.
The next proposition provides considerably more.
\begin{PROPB} Suppose that $\lambda$ is a regular uncountable cardinal satisfying $2^\lambda=\lambda^+$.

If $\mathbb P$ is a \textit{$\lambda^+$-cc} notion of forcing of size $\le 2^\lambda$,
then each of the following implies that $\mathbb P\in\mathbb C_\lambda$:
\begin{itemize}
\item $\mathbb P$ preserves the regularity of $\lambda$,
and is not ${}^\lambda\lambda$-bounding;\footnote{%
Recall that a notion of forcing is \emph{${}^\lambda\lambda$-bounding} if for every $g\in{}^\lambda\lambda\cap V^{\mathbb P}$,
there exists some $f\in{}^\lambda\lambda\cap V$ such that $g(\alpha)\le f(\alpha)$ for all $\alpha<\lambda$.}
\item $\mathbb P$ forces that $\lambda$ is a singular cardinal;
\item $\mathbb P$ forces that $\lambda$ is a singular ordinal satisfying  $2^{\cf(\lambda)}<\lambda$.
\end{itemize}
\end{PROPB}

It follows that Cohen, Prikry, Magidor and Radin forcing,\footnote{%
This includes variations likes the supercompact Prikry forcing and supercompact Magidor forcing (that singularizes a successor cardinal $\lambda$),
and includes the instances of Radin forcing that preserves the inaccessibility of $\lambda$.}
as well as some of the Namba-like forcings from the recent paper by Adolf, Apter, and Koepke \cite{aak},
are all members of the class under consideration.

\begin{THM}  Suppose that $\lambda=\lambda^{<\lambda}$ is a regular uncountable cardinal satisfying $2^\lambda=\lambda^+$.

Let $\kappa:=\lambda^+$. For all $\mathbb P\in\mathbb C_\lambda$, $V^{\mathbb P}\models \boxtimes^*(\kappa)+\diamondsuit(\kappa)$.
\end{THM}

One of the aspects that makes the proof of the preceding somewhat difficult is the fact that we do not assume that $\mathbb P$ is cofinality-preserving,
let alone assume that $\lambda$ remains a regular cardinal in $V^{\mathbb P}$.
In anticipation of constructions of $\kappa$-Souslin trees of a more involved nature, we will actually want to obtain a stronger principle than $\boxtimes^*(\kappa)$,
which we denote by $\p^*(T,\lambda)$. This principle dictates a strong form of Clause~(2) of Definition~\ref{xbox}, in which it can be shown that all but nonstationarily many $\alpha\in T$ will satisfy $|C|=|\lambda|$ for all $C\in\mathcal C_\alpha$.
When combined with the welcomed possibility that $\mathbb P$ forces that $\lambda$ is a singular cardinal (in which case  $V^{\mathbb P}\models\lambda^{<\lambda}>\lambda$),
ensuring Clause~(1) of Definition~\ref{xbox} becomes a burden.

\medskip

The definition of the weak square principle $\square^*_\lambda$ may be found in \cite[p.~283]{MR309729}.
Whether  $\gch+\square^*_{\lambda}$ is consistent with the nonexistence of $\lambda^+$-Souslin trees for a singular cardinal $\lambda$ is an open problem (see \cite{MR2162107}).
In \cite{rinot23}, generalizing a theorem of Gregory from \cite{MR485361}, we have shown that any model of such a consistency would have to satisfy that all stationary subsets of $E^{\lambda^+}_{\neq\cf(\lambda)}$ reflect.
Now, by \cite[$\S11$]{MR1838355} and \cite[$\S4$]{MR1942302}, if $\gch$ holds and  $\lambda$ is a $\lambda^+$-supercompact cardinal, then in the generic extension by Prikry forcing (using a normal measure on $\lambda$),
$\gch+\square^*_\lambda$ holds, and every stationary subset of $E^{\lambda^+}_{\neq\cf(\lambda)}$ reflects.
The next corollary (which follows from Propositions A,B and the Main Theorem) implies that nevertheless, this model contains a $\lambda^+$-Souslin tree.

\begin{CORA} Suppose that $\lambda$ is a strongly inaccessible cardinal satisfying $2^\lambda=\lambda^+$.

If $\mathbb P$ is a \textit{$\lambda^+$-cc} notion of forcing of size $\le2^\lambda$
that makes $\lambda$ into a singular cardinal, then $\mathbb P$ introduces a 
$\lambda^+$-Souslin tree.
\end{CORA}

Another corollary to the main result is the following:\footnote{But note that Proposition~\ref{p34} yields a stronger conclusion from a weaker arithmetic assumption.}
\begin{CORB} Suppose that $\theta<\lambda=\lambda^{<\lambda}$ are infinite regular cardinals, and $2^\lambda=\lambda^+$.

Then $\col(\theta,\lambda)$ introduces a $|\lambda|^+$-Souslin tree.
\end{CORB}

At the end of this short paper we shall also quickly deal with the case that $\kappa$ is a former inaccessible:
\begin{PROPC} Suppose that $\lambda^{<\lambda}=\lambda$ is an infinite cardinal,
and $\kappa>\lambda$ is a strongly inaccessible cardinal.
If $\mathbb P$ is a $({<\lambda})$-distributive, $\kappa$-cc notion of forcing,
collapsing $\kappa$ to $\lambda^+$, then $V^{\mathbb P}\models\boxtimes^*(\kappa)+\diamondsuit(\kappa)$ holds.
\end{PROPC}

Examples of notions of forcing satisfying the above requirements include the L\'evy collapse $\col(\lambda,{<\kappa})$, and the Silver collapse $\mathbb S(\lambda,{<\kappa})$.

\subsection*{Notation}
For an infinite cardinal $\lambda$, write  $\ch_\lambda$ for the assertion that $2^\lambda=\lambda^+$.
Next, suppose that $C,D$ are sets of ordinals.
Write $\acc(C):=\{\alpha\in C\mid \sup (C\cap\alpha) = \alpha>0 \}$, $\nacc(C) := C \setminus \acc(C)$,
and $\acc^+(C) := \{\alpha<\sup(C)\mid \sup (C\cap\alpha) = \alpha>0 \}$.
For any  $j < \otp(C)$, denote by $C(j)$ the unique element $\delta\in C$ for which $\otp(C\cap\delta)=j$.
For any ordinal $\sigma$, write
$\suc_\sigma(C) := \{ C(j+1)\mid j<\sigma\ \&\ j+1<\otp(C)\}$.
Write $D \sqsubseteq C$ iff there exists some ordinal $\beta$ such that $D = C \cap \beta$.
Write $D \sqsubseteq^* C$ iff there exists $\gamma < \sup(D)$ such that $D \setminus \gamma \sqsubseteq C \setminus \gamma$.
Write $D \sqx C$ iff (($D \sqsubseteq C$) or ($\cf(\sup(D))<\chi$)),
and write $D \sqx^* C$ iff (($D \sqsubseteq^* C$) or ($\cf(\sup(D)) <\chi$)).

\subsection*{Acknowledgment} We thank Gitik and Karagila for illuminating discussions on \cite{MR1357746} and \cite{MR1360144}.

\section{A construction of a Souslin tree from a weak hypothesis}

We begin this section by recalling the terminology relevant to trees and fixing some notation.

A \emph{tree} is a partially ordered set $(T,<_T)$ with the property that for every $x\in T$,
the downward cone $x_\downarrow:=\{ y\in T\mid y<_T x\}$ is well-ordered by $<_T$.
The \emph{height} of $x\in T$, denoted $\height(x)$,
is the order-type of $(x_\downarrow,<_T)$.
Then, the $\alpha^{\text{th}}$ level of $(T,<_T)$ is the set $T_\alpha:=\{x\in T\mid \height(x)=\alpha\}$.
We also write $T \restriction X := \{t \in T \mid \height(t) \in X \}$.
A tree  $(T,<_T)$ is said to be \emph{$\chi$-complete} if any $<_T$-increasing sequence of elements from $T$,
and of length ${<\chi}$, has an upper bound in $T$.
A tree $(T,<_T)$ is said to be \emph{normal} if for all ordinals $\alpha < \beta$
and every $x \in T_\alpha$, if $T_\beta\neq\emptyset$
then there exists some $y \in T_\beta$ such that $x <_T y$.
A tree $(T,<_T)$ is said to be \emph{splitting} if every node in $T$ admits at least two immediate successors.

Let $\kappa$ denote a regular uncountable cardinal.
A tree $(T,<_T)$ is a \emph{$\kappa$-tree} whenever $\{ \alpha\mid T_\alpha\neq\emptyset\}=\kappa$,
and $|T_\alpha|<\kappa$ for all $\alpha<\kappa$.
A subset $B\s T$ is a \emph{cofinal branch} if $(B,<_T)$ is linearly ordered and
$\{ \height(t)\mid t\in B\}=\{\height(t)\mid t\in T\}$.
A \emph{$\kappa$-Aronszajn tree} is a $\kappa$-tree with no cofinal branches.
A \emph{$\kappa$-Souslin tree} is a $\kappa$-Aronszajn tree that has no antichains of size $\kappa$.

A $\kappa$-tree is said to be \emph{binary} if it is a downward-closed subset of the complete binary tree ${}^{<\kappa}2$,
ordered by $\subset$.

\medskip

All the combinatorial principles considered in this paper are simplified instances of the proxy principle $\p^-(\kappa,\mu,\mathcal R,\theta,\mathcal S,\nu,\sigma,\mathcal E)$ that was introduced and studied in \cite{rinot22},\cite{rinot23},
but familiarity with those papers is not needed.

\begin{defn}
Suppose that $\kappa$ is a regular uncountable cardinal, and $S$ is a stationary subset of $\kappa$.
Let $\chi:=\min\{\cf(\alpha)\mid \alpha\in S\text{ limit}\}$.
The principle $\boxtimes^*(S)$ asserts the existence of a sequence $\langle \mathcal C_\alpha\mid\alpha<\kappa\rangle$ such that:
\begin{enumerate}
\item For every limit ordinal $\alpha<\kappa$:
\begin{itemize}
\item$\mathcal C_\alpha$ is a nonempty collection of club subsets of $\alpha$, with $\left| \mathcal C_\alpha \right| < \kappa$;
\item if $C \in \mathcal C_\alpha$ and $\bar\alpha$ is an accumulation point of $C$, then there exists some $D\in\mathcal C_{\bar\alpha}$ satisfying $D\mathrel{\sqx^*}C$;
\end{itemize}
\item For every cofinal $A\s\kappa$, there exist stationarily many $\alpha\in S$ such that $\sup(\nacc(C)\cap A)=\alpha$ for all $C \in \mathcal C_\alpha$.
\end{enumerate}
\end{defn}

\begin{remarks}
\begin{enumerate}[i.]
\item
While it is not entirely obvious, we omit the proof of the fact that $\boxtimes^*(S)$ for $S=\kappa$
coincides with Definition~\ref{xbox} given in the previous section.
\item Note that in general, it is not the case that $\boxtimes^*(S)$ implies $\boxtimes^*(T)$ for $T\supseteq S$.
\item By \cite[Theorem 1.8]{MR908147}, if $\kappa$ is a weakly compact cardinal, then $\boxtimes^*(S)$ fails for every stationary subset $S\s\kappa$.
\item In the langauge of \cite{rinot22},\cite{rinot23}, $\boxtimes^*(S)$ stands for $\p^-(\kappa,\kappa,{\sqx^*},1,\{S\},\kappa,1,\mathcal E_\kappa)$.
In particular, by \cite{rinot22}, $\boxtimes^*(E^\kappa_{\ge\chi})+\diamondsuit(\kappa)$ is consistent together with $\neg\diamondsuit(E^\kappa_{\ge\chi})$.
\end{enumerate}
\end{remarks}

Proposition A is a special case of the following:

\begin{prop}\label{prop24}
If $\kappa$ is a regular uncountable cardinal and $\chi < \kappa$ is a cardinal satisfying $\lambda^{<\chi}<\kappa$ for all $\lambda<\kappa$,
then $\boxtimes^*(E^\kappa_{\ge\chi})+\diamondsuit(\kappa)$ entails a normal, binary, splitting, $\chi$-complete $\kappa$-Souslin tree.
\end{prop}

\begin{proof} In~\cite[\S2]{rinot22}, we provided a construction of a
$\chi$-complete $\kappa$-Souslin using the stronger hypothesis $\boxtimes^-(E^\kappa_{\geq\chi})$.
Here we show that by taking some extra care in the construction,
we can get by assuming merely $\boxtimes^*(E^\kappa_{\ge\chi})$.

Let $\vec{\mathcal C}=\langle \mathcal C_\alpha\mid \alpha<\kappa\rangle$ be a witness to $\boxtimes^*(E^\kappa_{\ge\chi})$.
By \cite[\S2]{rinot22}, $\diamondsuit(\kappa)$ is equivalent to $\diamondsuit(H_\kappa)$, meaning that, in particular,
we may fix a sequence $\langle \Omega_\beta\mid\beta<\kappa\rangle$ satisfying the following:
For every $\Omega\s H_\kappa$ and $p\in H_{\kappa^+}$, there exists an elementary submodel $\mathcal M\prec H_{\kappa^+}$ containing $p$,
such that $\mathcal M\cap\kappa\in\kappa$ and $\mathcal M\cap\Omega=\Omega_{\mathcal M\cap\kappa}$.

Let $\lhd$ be some well-ordering of $H_\kappa$.
We shall recursively construct a sequence $\langle T_\alpha\mid \alpha<\kappa\rangle$ of levels
whose union will ultimately be the desired tree $T$.

Let $T_0:=\{\emptyset\}$, and for all $\alpha<\kappa$, let $T_{\alpha+1}:=\{ t{}^\smallfrown\langle0\rangle, t{}^\smallfrown\langle1\rangle\mid t\in T_\alpha\}$.

Next, suppose that $\alpha<\kappa$ is a nonzero limit ordinal,
and that $\langle T_\beta\mid \beta<\alpha\rangle$ has already been defined.
Constructing the level $T_\alpha$ involves deciding which branches
through $(T \restriction \alpha, \subset)$ will have their limits placed into the tree.
We need $T_\alpha$ to contain enough nodes to ensure that the tree is normal,
but we have to define $T_\alpha$ carefully, so that the resulting tree doesn't include large antichains.

Construction of the level $T_\alpha$ splits into two cases, depending on the value of $\cf(\alpha)$:

$\br$ If $\cf(\alpha)<\chi$, let $T_\alpha$ consist of the limits of all branches through $(T\restriction\alpha,\subset)$.
This construction ensures that the tree will be $\chi$-complete,
and as any branch through $(T\restriction\alpha,\subset)$ is determined by a subset of $T\restriction\alpha$ of cardinality $\cf(\alpha)$,
the arithmetic hypothesis ensures that $\left|T_\alpha\right| < \kappa$ at these levels.
Normality at these levels is verified by induction:
Fixing a sequence of ordinals of minimal order-type converging to $\alpha$
enables us to find, for any given $x \in T \restriction \alpha$,
a branch through $(T\restriction\alpha,\subset)$ containing $x$,
and the limit of such a branch will necessarily be in $T_\alpha$.

$\br$ Now suppose $\cf(\alpha)\ge\chi$.
Considering any $C \in \mathcal C_\alpha$,
the idea for ensuring normality at level $T_\alpha$ is to attach to each node $x \in T \restriction C$
some node $\mathbf b^C_x: \alpha \to 2$  above it,
and then let
\[
T_\alpha := \{ \mathbf b^C_x \mid C \in \mathcal C_\alpha, x \in T \restriction C \}.
\]

By the induction hypothesis, $|T_\beta|<\kappa$ for all $\beta<\alpha$, and by the choice of $\vec{\mathcal C}$, we have $|\mathcal C_\alpha|<\kappa$,
so that we are guaranteed to end up with $|T_\alpha|<\kappa$.

Let $C \in \mathcal C_\alpha$ and $x\in T\restriction C$ be arbitrary.
As $\mathbf b^C_x$ will be the limit of some branch through $(T\restriction\alpha,\subset)$ and above $x$,
it makes sense to describe $\mathbf b^C_x$ as the limit $\bigcup\rng(b^C_x)$ of a sequence $b^C_x\in\prod_{\beta\in C\setminus\dom(x)}T_\beta$ such that:
\begin{itemize}
\item $b^C_x(\dom(x))=x$;
\item $b_x^C(\beta') \subset b_x^C(\beta)$ for all $\beta'<\beta$ in $(C\setminus\dom(x))$;
\item $b^C_x(\beta)=\bigcup\rng(b^C_x\restriction\beta)$ for all $\beta\in\acc(C\setminus\dom(x))$.
\end{itemize}
Of course, we have to define $b^C_x$ carefully, so that the resulting tree doesn't include large antichains.
We do this by recursion:

Let $b^C_x(\dom(x)):=x$.
Next, suppose $\beta^-<\beta$ are consecutive points of $(C\setminus\dom(x))$,
and $b^C_x(\beta^-) \in T_{\beta^-}$ has already been identified.
In order to decide $b^C_x(\beta)$, we advise with the following set:
$$Q^{C, \beta}_x := \{ t\in T_\beta\mid \exists s\in \Omega_{\beta}[ (s\cup b^C_x(\beta^-))\s t]\}.$$
Now, consider the two possibilities:
\begin{itemize}
\item If $Q^{C,\beta}_x \neq \emptyset$, then let $b^C_x(\beta)$ be its $\lhd$-least element;
\item Otherwise, let $b^C_x(\beta)$ be the $\lhd$-least element of $T_\beta$ that extends $b^C_x(\beta^-)$.
Such an element must exist, as the level $T_\beta$ was constructed so as to preserve normality.
\end{itemize}

Finally, suppose $\beta \in \acc(C\setminus\dom(x))$
and $b^C_x\restriction\beta \in\prod_{\delta \in C \cap \beta \setminus\dom(x)}T_\delta$ has already been defined.
As promised, we let $b^C_x(\beta):=\bigcup\rng(b^C_x\restriction\beta)$.
It is clear that $b^C_x(\beta) \in {}^\beta 2$,
but we need more than that:

\begin{claim}\label{coherence}
$b^C_x (\beta) \in T_\beta$.
\end{claim}

\begin{proof}
We consider two cases, depending on the value of $\cf(\beta)$:

$\br$ If $\cf(\beta) < \chi$, then $T_\beta$ was constructed to consist of
the limits of all branches through $(T\restriction\beta,\subset)$,
including the limit of the branch $b^C_x\restriction\beta$, which is $b^C_x (\beta)$.

$\br$ Now suppose $\cf(\beta)\ge\chi$.
In this case, since $\beta \in \acc(C)$ and
by $\sqleft{\chi}^*$-coherence of the $\boxtimes^*(E^\kappa_{\ge\chi})$-sequence
$\vec{\mathcal C}$,
we can pick some $D \in \mathcal C_\beta$ such that $D \sqsubseteq^* C$,
and we can thus fix some $\gamma \in (C \cap D) \setminus \dom(x)$
such that $D \setminus \gamma = C \cap \beta \setminus \gamma$.
Put $d:=D \setminus \gamma$ and $y := b^C_x(\gamma)$.
Then $y\in T_\gamma$
and $\dom(b^D_y) = d = C \cap \beta \setminus \gamma = \dom(b^C_x) \cap \beta \setminus \gamma$.

It suffices to prove that $b^C_x \restriction d =b^D_y$,
as this will imply that $b^C_x(\beta)=\bigcup\rng(b^D_y)=\mathbf{b}^D_y \in T_\beta$,
since the limit of a branch is determined by any of its cofinal segments.
Thus, we prove by induction that
for every $\delta \in d$,
the value of $b^D_y(\delta)$ was determined in exactly the same way as $b^C_x(\delta)$:
\begin{itemize}
\item Clearly, $b^D_y(\min(d)) = y = b^C_x(\min(d))$, since $\min(d) = \gamma = \dom(y)$.
\item Suppose $\delta^-<\delta$ are successive points of $d$.
Notice that the definition of $Q^{C, \delta}_x$ depends only on
$b^C_x(\delta^-)$, $\Omega_\delta$, and $T_\delta$,
and so if $b^C_x (\delta^-) = b^D_y(\delta^-)$,
then $Q^{C, \delta}_x = Q^{D, \delta}_y$,
and hence $b^C_x(\delta) = b^D_y(\delta)$.

\item For $\delta \in \acc(d)$:
If the sequences are identical up to $\delta$, then their limits must be identical. \qedhere
\end{itemize}
\end{proof}

This completes the definition of $b^C_x$ for each $C \in \mathcal C_\alpha$ and each $x\in T\restriction C$,
and hence of the level $T_\alpha$.

Having constructed all levels of the tree, we then let
\[
T := \bigcup_{\alpha < \kappa} T_\alpha.
\]

Notice that for every $\alpha < \kappa$,
$T_\alpha$ is a subset of $^\alpha2$ of size $< \kappa$.
Altogether, $(T, \subset)$ is a normal, binary, splitting, $\chi$-complete $\kappa$-tree.

Our next task is proving that $(T,\subset)$ is $\kappa$-Souslin.
As any splitting $\kappa$-tree with no antichains of size $\kappa$ also has no chains of size $\kappa$,
it suffices to prove Claim~\ref{c233} below. For this, we shall need
the following:

\begin{claim}\label{c232}
Suppose that $A \subseteq T$ is a maximal antichain.
Then the set
\[
B := \{ \beta <\kappa \mid  A\cap(T\restriction\beta)= \Omega_\beta\text{ is a maximal antichain in }T\restriction\beta \}.
\]
is a stationary subset of $\kappa$.
\end{claim}
\begin{proof}
Let $D\s\kappa$ be an arbitrary club.
We must show that $D \cap B \neq \emptyset$.
Put $p := \{A,T,D\}$ and $\Omega := A$.
By our choice of the sequence $\langle \Omega_\beta\mid\beta<\kappa\rangle$,
pick $\mathcal M\prec H_{\kappa^{+}}$
with $p\in\mathcal M$ such that $\beta :=\mathcal M\cap\kappa$ is in $\kappa$ and $\Omega_\beta=\mathcal M \cap A$.
Since $D\in\mathcal M$ and $D$ is club in $\kappa$, we have $\beta\in D$. We claim that $\beta\in B$.

For all $\alpha<\beta$, by $\alpha,T\in \mathcal M$, we have $T_\alpha\in \mathcal M$,
and by $\mathcal M\models |T_\alpha|<\kappa$,
we have $T_\alpha\s \mathcal M$. So $T\restriction\beta\s \mathcal M$.
As $\dom(z)\in \mathcal M$ for all $z \in T\cap \mathcal M$,
we conclude that $T\cap \mathcal M=T\restriction\beta$.
Thus, $\Omega_\beta=A\cap(T\restriction\beta)$.
As $H_{\kappa^{+}}\models A\text{ is a maximal antichain in }T$,
it follows by elementarity that
$\mathcal M \models A\text{ is a maximal antichain in }T$.
Since $T\cap \mathcal M=T\restriction \beta$,
we get that $A\cap(T\restriction\beta)$ is a maximal antichain in $T\restriction\beta$.
\end{proof}

\begin{claim}\label{c233}
Suppose that $A \subseteq T$ is a maximal antichain. Then $|A|<\kappa$.
\end{claim}

\begin{proof}
Let $A \subseteq T$ be a maximal antichain.
By Claim~\ref{c232},
$$B := \{ \beta <\kappa \mid  A\cap(T\restriction\beta)= \Omega_\beta\text{ is a maximal antichain in }T\restriction\beta \}$$
is a stationary subset of $\kappa$.
Thus we apply $\boxtimes^*(E^\kappa_{\ge\chi})$ to obtain
an ordinal $\alpha\in E^\kappa_{\ge\chi}$ satisfying
\[
\sup (\nacc(C) \cap B) = \alpha
\]
for every $C \in \mathcal C_\alpha$.

We shall prove that $A \subseteq T \restriction \alpha$,
from which it follows that $|A| \leq |T \restriction \alpha| <\kappa$.

To see that $A \subseteq T \restriction \alpha$, consider any $z \in T \restriction (\kappa \setminus \alpha)$,
and we will show that $z \notin A$ by finding some element of $A \cap (T \restriction \alpha)$ compatible with $z$.

Since $\dom(z) \geq \alpha$,
we can let $y := z \restriction \alpha$. Then $y \in T_\alpha$ and $y \subseteq z$.
By construction, since $\cf(\alpha)\ge\chi$,
we have $y = \mathbf b^C_x = \bigcup_{ \beta \in C \setminus \dom(x) } b^C_x (\beta)$
for some $C \in \mathcal C_\alpha$ and some $x \in T \restriction C$.
Fix $\beta \in \nacc(C) \cap B$ with $\dom(x) < \beta < \alpha$.
Denote $\beta^-:=\sup(C\cap\beta)$. Then $\beta^- < \beta$ are consecutive points of $C \setminus \dom(x)$.
Since $\beta \in B$,
we know that $\Omega_\beta = A \cap (T \restriction \beta)$ is a maximal antichain in $T \restriction \beta$,
and hence there is some $s \in \Omega_\beta$ compatible with $b^C_x(\beta^-)$,
so that by normality of the tree, $Q^{C, \beta}_x \neq \emptyset$.
It follows that we chose $b^C_x(\beta)$ to extend some $s \in \Omega_\beta$.
Altogether, $s \subseteq b^C_x(\beta) \subset \mathbf b^C_x = y \subseteq z$.
Since $s$ is an element of the antichain $A$, the fact that $z$ extends $s$ implies that $z \notin A$.
\end{proof}

So $(T,\subset)$ is a normal, binary, splitting, $\chi$-complete $\kappa$-Souslin tree.
\end{proof}

\section{Main Results}

We begin this section by proving Proposition B:

\begin{prop}\label{p31} Suppose that $\lambda$ is a regular uncountable cardinal satisfying $2^\lambda=\lambda^+$.

If $\mathbb P$ is a \textit{$\lambda^+$-cc} notion of forcing of size $\le2^\lambda$,
then each of the following implies that $\mathbb P\in\mathbb C_\lambda$:
\begin{enumerate}
\item $\mathbb P$ preserves the regularity of $\lambda$,
and is not ${}^\lambda\lambda$-bounding;
\item $\mathbb P$ forces that $\lambda$ is a singular cardinal;
\item $\mathbb P$ forces that $\lambda$ is a singular ordinal, satisfying  $2^{\cf(\lambda)}<\lambda$.
\end{enumerate}
\end{prop}
\begin{proof} Let $G$ denote a $\mathbb P$-generic filter.

(1) Work in $V[G]$.
Since $\mathbb P$ is not ${}^\lambda\lambda$-bounding,
let us pick $g\in{}^\lambda\lambda$  with the property that for every $f\in{}^\lambda\lambda\cap V$,
there exists some $\alpha<\lambda$ with $f(\alpha) < g(\alpha)$. Since $\lambda$ is regular and uncountable,
the set $\Lambda:=\{\zeta<\lambda\mid g[\zeta]\s\zeta\}$ is a club in $\lambda$.
To see that $\Lambda$ works, let $f\in{}^\lambda\lambda\cap V$ be arbitrary.
In $V$, let $f':\lambda\rightarrow\lambda$ be a strictly increasing function such that $f(\xi)\le f'(\xi)$ for all $\xi<\lambda$.
Now, back in $V[G]$, pick $\alpha<\lambda$ such that $f'(\alpha) < g(\alpha)$.
Put $\xi:=\sup(\Lambda\cap (\alpha+1))$ and $\zeta:=\min(\Lambda\setminus(\xi+1))$.
Then $\xi \in \Lambda$ because $\Lambda$ is a club in $\lambda$ containing $0$, and
$f(\xi)\le f'(\xi) \leq f'(\alpha) < g(\alpha)<\zeta$ because $\alpha<\zeta$ and $\zeta\in\Lambda$.

(2) Note that as $\mathbb P$ forces that $\lambda$ is a singular cardinal,
$\lambda$ cannot be a successor cardinal in the ground model.
Also note that by the \textit{$\lambda^+$-cc} of $\mathbb P$, we have $(\lambda^+)^{V[G]}=(\lambda^+)^V$.
Thus, by \cite[Theorem 2.0]{MR1360144}, since
$\lambda$ is inaccessible, $(\lambda^+)^{V[G]}=(\lambda^+)^V$ and $2^\lambda=\lambda^+$,
there exists a cofinal $\Lambda\s \lambda$ in $V[G]$ such that $\sup(\Lambda\setminus C)<\lambda$ for every club $C$ in $\lambda$ from $V$.
Thus, to see that $\Lambda$ works, let $f\in{}^\lambda\lambda\cap V$ be arbitrary.
Consider the club $C:=\{\zeta<\lambda\mid f[\zeta]\s\zeta\}$. As $C\in V$, pick $\xi\in\Lambda$ such that $\Lambda\setminus\xi\s C$.
Put $\zeta:=\min(\Lambda\setminus(\xi+1))$. Then $\zeta\in C$ and hence $f(\xi)<\zeta$.

(3) Let $\kappa:=(\lambda^+)^V$. Work in $V[G]$. Let $\theta:=\cf(\lambda)$.
By $2^\theta<\lambda$, we have $(2^\theta)^+\le|\lambda|^+$. Since $\mathbb P$ is \textit{$\kappa$-cc},
we have $|\lambda|^+=\kappa$. In particular, $\cf(\kappa)=\kappa\ge(2^\theta)^+$. Also, by $2^\theta<\lambda$, we have $\theta^+<\lambda$.
Then, the two conditions of \cite[Proposition 2.1]{MR1357746} are satisfied and
hence there exists a cofinal $\Lambda\s \lambda$ such that $\sup(\Lambda\setminus C)<\lambda$ for every club $C$ in $\lambda$ from $V$.
Thus, we are done as in the previous case.
\end{proof}

\begin{lemma}\label{l2} Suppose that $\lambda$ is a regular uncountable cardinal, and $\mathbb P\in\mathbb C_\lambda$.

In $V^{\mathbb P}$, there exists a club $\Lambda\s\lambda$ of order-type $\cf(\lambda)$,
such that for every function $f\in{}^\lambda\lambda\cap V$, $\sup\{\xi\in\Lambda\mid f(\xi)<\min(\Lambda\setminus(\xi+1))\}=\lambda$.
\end{lemma}
\begin{proof} Let $G$ be $\mathbb P$-generic, and work in $V[G]$.
Pick $\Lambda$ as in Clause~(2) of Definition~\ref{theclass}.
Let $\mathring{\Lambda}$ be a club in $\lambda$ of order-type $\cf(\lambda)$ such that $\nacc(\mathring\Lambda)\s\Lambda$.
To see that $\mathring\Lambda$ works, fix an arbitrary $f\in{}^\lambda\lambda\cap V$ and an arbitrary $\iota<\lambda$.
We shall find $\mathring\xi\in\mathring\Lambda$ above $\iota$ satisfying $f(\mathring\xi)<\min(\mathring\Lambda\setminus(\mathring\xi+1))$.

Fix a large enough $\epsilon\in\nacc(\mathring\Lambda)$ such that $\sup(\mathring\Lambda\cap\epsilon)>\iota$.
In $V$, using the regularity of $\lambda$, pick a strictly increasing function $f':\lambda\rightarrow\lambda$ such that $f'(\alpha)\ge\max\{f(\alpha),\epsilon\}$ for all $\alpha<\lambda$.
By the choice of $\Lambda$, let us pick $\xi\in\Lambda$ such that $f'(\xi)<\min(\Lambda\setminus(\xi+1))$.
In particular, $\epsilon<\zeta$, where $\zeta:=\min(\Lambda\setminus(\xi+1))$.
As $\epsilon\in\nacc(\mathring\Lambda)\s\Lambda$, and as $\xi<\zeta$ are two successive elements of $\Lambda$ with $\epsilon<\zeta$, we have $\epsilon\le\xi$.
Put $\mathring\xi:=\sup(\mathring\Lambda\cap(\xi+1))$ and $\mathring\zeta:=\min(\mathring\Lambda\setminus(\mathring\xi+1))$.
By $\xi\ge\epsilon$ and $\sup(\mathring\Lambda\cap\epsilon)>\iota$, we have $\mathring\xi>\iota$.
By $\mathring\zeta\in\nacc(\mathring\Lambda)\setminus(\mathring\xi+1)\s\Lambda\setminus(\xi+1)$,
we have $\mathring\zeta\ge\zeta$.
Altogether:
\begin{itemize}
\item $\iota<\mathring\xi\le\xi<\zeta\le\mathring\zeta$, and
\item $f(\mathring\xi)\le f'(\mathring\xi)\le f'(\xi)<\zeta\le\mathring\zeta=\min(\mathring\Lambda\setminus(\mathring\xi+1))$,
\end{itemize}
So that $\mathring\xi$ is as sought.
\end{proof}

\begin{defn}
Suppose that $T$ is a stationary subset of a regular uncountable cardinal $\kappa$,
and $\xi\le\kappa$ is an ordinal.

The principle $\p^*(T,\xi)$ asserts the existence of a sequence $\left< \mathcal C_\alpha \mid \alpha < \kappa \right>$
such that:
\begin{enumerate}
\item For every limit ordinal $\alpha<\kappa$:
\begin{itemize}
\item$\mathcal C_\alpha$ is a nonempty collection of club subsets of $\alpha$ of order-type $\le\xi$, with $\left| \mathcal C_\alpha \right| < \kappa$;
\item if $C \in \mathcal C_\alpha$ and $\bar\alpha$ is an accumulation point of $C$, then $C\cap\bar\alpha\in\mathcal C_{\bar\alpha}$;
\end{itemize}
\item For every cofinal subset $A\s\kappa$, all but nonstationarily many $\alpha \in T$ satisfy:
\begin{itemize}
\item $|\mathcal C_\alpha|=1$, say, $\mathcal C_\alpha=\{C_\alpha\}$, and
\item
$\sup\{ \beta \in C_\alpha \mid \suc_{\sigma} (C_\alpha \setminus \beta) \subseteq A \} = \alpha$ for every $\sigma<\otp(C_\alpha)$.
\end{itemize}
\end{enumerate}
\end{defn}

\begin{remarks}
\begin{enumerate}[i.]
\item Note that for stationary sets $S\s T\s\kappa$, $\p^*(T,\xi)\implies \p^*(S,\xi)\implies \p^*(S,\kappa)\implies \boxtimes^*(S)$,
and that $\p^*(S,\kappa)\implies \boxtimes^*(\kappa)$.
Thus, the remaining theorems in this paper will focus on establishing $\p^*(T,\xi)+\diamondsuit(\kappa)$
for some stationary $T \subseteq \kappa$ and some $\xi\leq\kappa$ in various forcing scenarios.
\item By arguments that may be found in \cite[$\S6$]{rinot20}, $\p^*(S,\kappa)+\diamondsuit(\kappa)$ entails the existence of a $\kappa$-Souslin tree which is moreover \emph{free}.
We do not know whether $\boxtimes^*(\kappa)+\diamondsuit(\kappa)$ suffices for this application.
\item In the langauge of \cite{rinot22},\cite{rinot23}, $\p^*(T,\xi)$ stands for $\p^-(\kappa,\kappa,{\sq},1,\ns_\kappa\restriction T,2,{<\infty},\mathcal E_\xi)$.
\end{enumerate}
\end{remarks}

We now arrive at the main result of this paper:

\begin{thm}\label{thm46} Suppose that $\lambda^{<\lambda}=\lambda$ is a regular uncountable cardinal, and $\ch_\lambda$ holds.

Let $\kappa:=\lambda^+$ and $T:= E^{\kappa}_\lambda$.
For every $\mathbb P\in\mathbb C_\lambda$, $V^{\mathbb P}\models \p^*(T,\lambda)+\diamondsuit(\kappa)$.
\end{thm}
\begin{proof}
By $\ch_\lambda$ and the main result of \cite{Sh:922}, $\diamondsuit(\kappa)$ holds.
By \cite[\S2]{rinot22}, $\diamondsuit(\kappa)$ is equivalent to $\diamondsuit(H_{\kappa})$,
meaning that, in particular,
we may fix a sequence $\langle \Omega_\beta\mid\beta<\kappa\rangle$ satisfying the following:
For every $\Omega\s H_{\kappa}$ and $p\in H_{\kappa^+}$, there exists an elementary submodel $\mathcal M\prec H_{\kappa^+}$ containing $p$,
such that $\mathcal M\cap\kappa\in\kappa$ and $\mathcal M\cap\Omega=\Omega_{\mathcal M\cap\kappa}$.

Let $\mathbb P \in \mathbb C_\lambda$ be arbitrary.
By $|\mathbb P|\le|H_{\kappa}|$, we may assume that $\mathbb P\s H_{\kappa}$.
Let $G\s\mathbb P$ be $V$-generic, and work in $V[G]$.
Since $\mathbb P$ is a \textit{$\kappa$-cc} notion of forcing,
$\kappa$ remains a regular cardinal and $T$ remains stationary in $V[G]$.

For all $\beta<\kappa$ such that $\Omega_\beta$ happens to be a $\mathbb P$-name for a subset of $\beta$,
let  $Z_\beta$ denote its interpretation by $G$. For all other $\beta<\kappa$, let $Z_\beta:=\emptyset$.
\begin{claim}\label{c743} For every $A\in\mathcal P^{V[G]}(\kappa)$, there exists some $X_A\in\mathcal P^{V}(\kappa)$ such that:
\begin{enumerate}
\item $V\models X_A\text{ is stationary}$;
\item $V[G]\models X_A\s \{\beta<\kappa\mid Z_\beta=A\cap\beta\}$.
\end{enumerate}
\end{claim}
\begin{proof}
Fix an arbitrary $A\s\kappa$ in $V[G]$, and let $\name{A}$ be a $\mathbb P$-name for $A$.

Work in $V$.  For every $\alpha < \kappa$, put
$O_\alpha := \{ p \in \mathbb P \mid p \forces_\mathbb P \check \alpha \in \name A \}$,
and choose some maximal antichain $A_\alpha \subseteq O_\alpha$.
Then $\Omega := \bigcup\{ \{\check\alpha\}\times A_{\alpha}\mid \alpha<\kappa\}$ is a \emph{nice name} for $A$. In particular:
\begin{itemize}
\item $A_{\alpha}$ is an antichain in $\mathbb P$ for all $\alpha<\kappa$;
\item $\forces_\mathbb P \name A=\Omega$.
\end{itemize}
Since $\mathbb P\s H_{\kappa}$ is a \textit{$\kappa$-cc} notion of forcing, we also have:
\begin{itemize}
\item $A_{\alpha}\in H_{\kappa}$ for all $\alpha<\kappa$.
\end{itemize}

Altogether, $\Omega\s H_{\kappa}$, $|\Omega|\le\kappa$, and hence $\Omega \in H_{\kappa^+}$.
Let $$X_A:=\{ \mathcal M\cap\kappa\mid \mathcal M\prec H_{\kappa^+}, \Omega\in \mathcal M, \mathcal M\cap\kappa\in\kappa, \mathcal M\cap\Omega=\Omega_{\mathcal M\cap\kappa}\}.$$

(1) To see that $X_A$ is stationary, let $D$ be an arbitrary club in $\kappa$. Put $p:=\{\Omega,D\}$. By the fixed witness to $\diamondsuit(H_{\kappa})$,
we may now pick an elementary submodel $\mathcal M\prec H_{\kappa^+}$ such that $p\in\mathcal M$, $\mathcal M\cap\kappa\in\kappa$ and $\mathcal M\cap\Omega=\Omega_{\mathcal M\cap\kappa}$.
Denote $\beta:=\mathcal M\cap \kappa$. By $D\in\mathcal M$, we have $\beta\in D\cap X_A$.

(2) Let $\beta\in X_A$ be arbitrary, with witnessing model $\mathcal M$.
By $\Omega\in\mathcal M$ and $|A_\alpha|<\kappa$ for all $\alpha<\kappa$, we have $\Omega_\beta = \mathcal M\cap\Omega=\bigcup\{\{\check\alpha\}\times A_{\alpha}\mid \alpha<\beta\}$.
So $\Omega_\beta$ is a nice name and $\forces_\mathbb P \Omega_\beta = \Omega \cap \check \beta = \name A \cap \check \beta$.
That is, $\Omega_\beta$ is a $\mathbb P$-name whose interpretation in $V[G]$ is $A\cap\beta$,
and hence $V[G]\models Z_\beta=A\cap\beta$.
\end{proof}

Since $\mathbb P$ is a \textit{$\kappa$-cc} notion of forcing, every stationary subset of $\kappa$ from $V$ remains stationary in $V[G]$,
and so it follows from the previous claim that $\diamondsuit(\kappa)$ holds in $V[G]$.
Thus, we are left with verifying that $\p^*(T,\lambda)$ holds in $V[G]$.

\medskip

Work back in $V$.
By $\lambda^{<\lambda}=\lambda$, let us fix
a sequence of injections $\langle \varrho_\alpha:\alpha\rightarrow\lambda\mid \alpha<\kappa\rangle$ with the property
that for all $\delta<\kappa$, we have $|\{ \varrho_\alpha\restriction\delta\mid \alpha<\kappa\}|<\kappa$.\footnote{Note that the existence of such a sequence is equivalent to the existence of a special $\lambda^+$-Aronszajn tree,
which is a well-known consequence of $\lambda^{<\lambda}=\lambda$ (cf.~\cite{MR0039779} or \cite[Theorem 7.1]{MR776625}).}
For every $\alpha\in T$, let $\pi_\alpha:\lambda\rightarrow E^\alpha_{<\lambda}$ denote the monotone enumeration of some club in $\alpha$.
Notice that $\rng(\pi_\alpha) \cap T = \emptyset$ for every $\alpha \in T$.

By $\ch_\lambda$, let us fix an enumeration $\{X^j\mid j<\kappa\}$ of all bounded subsets of $\kappa$ such that each element appears cofinally often.

Next, for every nonempty $x\in[\kappa]^{<\lambda}$ and every $\beta\in x\cap T$, define $h_{x,\beta}:\otp(x)\rightarrow [\beta]^{1}\cup [\beta]^{(1+\otp(x))}$ as follows:

Put $h_{x,\beta}(0):=\{\sup(x\cap\beta)\}$.
Now, if $i\in(0,\otp(x))$ and $h_{x,\beta}\restriction i$ has already been defined,
write $\epsilon_{x,\beta}(i):=\sup(\bigcup_{i'<i}h_{x,\beta}(i'))+1$ and $j:=x(i)$.
Then:

\begin{itemize}
\item[$\br$] If $\sup(X^j\cap\beta)\neq\beta$, then let $h_{x,\beta}(i):=\{\epsilon_{x,\beta}(i)\}$;
\item[$\br$] If  $\sup(X^j\cap\beta)=\beta$, then let $h_{x,\beta}(i):=\{\epsilon_{x,\beta}(i)\}\cup\suc_{\otp(x)}(X^j\setminus \epsilon_{x,\beta}(i))$.
\end{itemize}

Note that all of the above are well-defined, thanks to the fact that $\cf(\beta)=\lambda>|x|$.

Finally, for every nonempty $x\in[\kappa]^{<\lambda}$, let $\overline{x}$ be the ordinal closure of $$x\cup\bigcup\{\bigcup\rng(h_{x,\beta})\mid \beta\in x\cap T, \beta\neq \min(x)\}.$$
Clearly, $\min(\overline x)=\min(x)$, $\max(\overline x)=\sup(x)$, and $|\overline x|\le\max\{|x|,\aleph_0\}$.
In particular, $\otp(\overline x)<\lambda$.
In addition, for all $\beta\in x\cap T$ such that $\gamma:=\sup(x\cap\beta)$ is nonzero,
we have that $\overline{x}\cap[\gamma,\beta)$ is equal to the ordinal closure of $\bigcup\rng(h_{x,\beta})$.

\medskip

Work in $V[G]$.
By $\mathbb P \in \mathbb C_\lambda$, let $\Lambda$ be the club given by Lemma~\ref{l2}.
For all $\alpha\in T$, let:
$$C_\alpha:= \bigcup \left\{ \overline{[\pi_\alpha(\xi),\pi_\alpha(\zeta))\cap(\{\pi_\alpha(\xi)\}\cup \varrho_\alpha^{-1}[\zeta])}\mid \xi\in\Lambda, \zeta=\min(\Lambda\setminus(\xi+1)) \right\}.$$
Notice that $\pi_\alpha[\Lambda]\s C_\alpha\s \alpha$, so that $C_\alpha$ is a club in $\alpha$.

Before stating the next claim, let us remind the reader that we do not assume that $\mathbb P$ is cofinality-preserving.
\begin{claim}
For every $\alpha \in T$:
\begin{enumerate}
\item $\acc(C_\alpha)\cap T=\emptyset$;
\item $\otp(C_\alpha)\le\lambda$.
\end{enumerate}
\end{claim}
\begin{proof} (1) Suppose that $\bar\alpha\in\acc(C_\alpha)$. Put $\xi':=\sup\{ \xi\in\Lambda\mid \pi_\alpha(\xi)\le\bar\alpha\}$.

$\br$ If $\bar\alpha=\pi_\alpha(\xi')$, then $\bar\alpha\in\rng(\pi_\alpha)$ and hence $\bar\alpha\notin T$.

$\br$ Otherwise, letting $\zeta':=\min(\Lambda\setminus(\xi'+1))$, we get that $\bar\alpha$ is an accumulation point of $$\overline{[\pi_\alpha(\xi'),\pi_\alpha(\zeta'))\cap (\{\pi_\alpha(\xi')\}\cup \varrho_\alpha^{-1}[\zeta'])},$$
which is a set of ordinals from $V$ of size $<\lambda$, and hence $\cf^V(\bar\alpha)<\lambda$,
so that $\bar\alpha \notin T$.

(2) We prove by induction on $\xi\in\Lambda$ that $\otp(C_\alpha\cap\pi_\alpha(\xi))<\lambda$ for all $\xi\in\Lambda$:

$\br$ For $\xi=\min(\Lambda)$, we have $C_\alpha\cap\pi_\alpha(\xi)=\emptyset$.

$\br$ Suppose that $\xi\in\Lambda$ and $\otp(C_\alpha\cap\pi_\alpha(\xi))<\lambda$.
Put $\zeta:=\min(\Lambda\setminus(\xi+1))$. Then $$\otp(C_\alpha\cap\pi_\alpha(\zeta))=\otp(C_\alpha\cap\pi_\alpha(\xi))+\otp\left(\overline{[\pi_\alpha(\xi),\pi_\alpha(\zeta))\cap(\{\pi_\alpha(\xi)\}\cup \varrho_\alpha^{-1}[\zeta])}\right),$$
so that $\otp(C_\alpha\cap\pi_\alpha(\zeta))$ is the sum of two ordinals $<\lambda$.
As $\lambda$ is a cardinal in $V$, it is an additively indecomposable ordinal, and hence the sum of the two is still $<\lambda$.

$\br$ Suppose that $\xi\in\acc(\Lambda)$ and $\otp(C_\alpha\cap\pi_\alpha(\epsilon))<\lambda$ for all $\epsilon\in\Lambda\cap\xi$.
That is, $$\{\otp(C_\alpha\cap\pi_\alpha(\epsilon))\mid \epsilon\in\Lambda\cap\xi\}\s \lambda.$$
As $\otp(C_\alpha\cap\pi_\alpha(\xi))=\sup\{\otp(C_\alpha\cap\pi_\alpha(\epsilon))\mid \epsilon\in\Lambda\cap\xi\}$,
and $|\Lambda\cap\xi|<|\Lambda|=\cf(\lambda)$, we infer that $\otp(C_\alpha\cap\pi_\alpha(\xi))<\lambda$.
\end{proof}

\begin{claim}\label{c8103} For every $\delta<\kappa$, we have $|\{ C_\alpha\cap\delta\mid \alpha\in T\}|<\kappa$.
\end{claim}
\begin{proof} Suppose not, and let $\delta$ be the least counterexample.
Pick a subset $A\s T\setminus(\delta+1)$ of size $\kappa$ such that:
\begin{itemize}
\item $\alpha\mapsto C_\alpha\cap\delta$ is injective over $A$;
\item $\alpha\mapsto\varrho_\alpha\restriction\delta$ is constant over $A$;
\item $\alpha\mapsto\sup\{ \xi\in\Lambda\mid \pi_\alpha(\xi)\le\delta\}$ is constant over $A$, with value, say, $\xi'$;
\item $\alpha\mapsto\pi_\alpha \restriction (\xi'+1)$ is constant over $A$.\footnote{Here, we use two facts: (1) If $\pi_\alpha(\xi')\le\delta$,
then $\pi_\alpha\restriction(\xi'+1)$ is the increasing enumeration of an element of $[\delta+1]^{<\lambda}\cap V$.
(2) By $V\models |[\delta+1]^{<\lambda}|=\lambda<\kappa$, we have $V[G]\models|[\delta+1]^{<\lambda}\cap V|<\kappa$.}
\end{itemize}

Clearly, $\alpha\mapsto C_\alpha\cap\pi_\alpha(\xi')$ is constant over $A$.
Write $\zeta':=\min(\Lambda\setminus(\xi'+1))$. Then it must be the case that $\alpha\mapsto C_\alpha \cap [\pi_\alpha(\xi'), \delta)$ is injective over $A$. That is,
$$\alpha\mapsto \overline{[\pi_\alpha(\xi'),\pi_\alpha(\zeta'))\cap (\{\pi_\alpha(\xi')\}\cup \varrho_\alpha^{-1}[\zeta'])}\cap\delta$$
is injective over $A$, contradicting the fact that the right-hand side of the above mapping is an element of $[\delta]^{<\lambda}\cap V$,
and $V\models|[\delta]^{<\lambda}|<\kappa$.
\end{proof}

Since $\mathbb P$ is a \textit{$\kappa$-cc} notion of forcing, every club $D\s\kappa$ from $V[G]$
contains a club $D'\s D$ from $V$. Hence, whenever we talk about clubs in $\kappa$, we may as well assume that they lie in $V$.

\begin{claim}\label{c742} For every cofinal subset $X\s\kappa$ from $V$,
there exists a club $D_X\s\kappa$, such that for every $\alpha\in T\cap D_X$ and every $\sigma<\lambda$, we have $\sup\{\epsilon\in C_\alpha\mid \suc_\sigma(C_\alpha\setminus\epsilon)\s  X\}=\alpha$.
\end{claim}
\begin{proof}
Work in $V$. Given a cofinal subset $X\s\kappa$,
define $f:\kappa\rightarrow\kappa$ by stipulating
$$f(\beta):=\min\{ j<\kappa\mid X^j=X\cap\beta\ \&\ j\ge\beta\}.$$

Consider the club $D_X:=\{ \alpha<\kappa \mid  f[\alpha]\s\alpha\}\cap\acc^+(\acc^+(X)\cap T)$.
Let $\alpha\in T\cap D_X$ be arbitrary.
Let $\tau<\alpha$ and $\sigma<\lambda$ be arbitrary. We shall prove the existence of some $\epsilon\in C_\alpha\setminus\tau$ such that $\suc_\sigma(C_\alpha\setminus\epsilon)\s X$.

Let $\varrho_\alpha^{-01}:\lambda\rightarrow\alpha$ denote a pseudoinverse of the injection $\varrho_\alpha:\alpha\rightarrow\lambda$, as follows:
$$\varrho_\alpha^{-01}(\xi):=\begin{cases}\varrho_\alpha^{-1}(\xi),&\text{if }\xi\in\rng(\varrho_\alpha);\\
0,&\text{otherwise}.\end{cases}$$
Define $f_0,f_1,f_2,f_3:\lambda\rightarrow\lambda$ by stipulating:
\begin{itemize}
\item $f_0(\xi):= \min(\varrho_\alpha[(\acc^+(X)\cap T \cap \alpha)\setminus \pi_\alpha(\xi)])$;
\item $f_1(\xi):= \varrho_\alpha(f(\varrho_\alpha^{-01}(\xi)))$;
\item $f_2(\xi):=\min\{\zeta<\lambda\mid \varrho_\alpha^{-01}(\xi)\le \pi_\alpha(\zeta)\}$;
\item $f_3(\xi):=\min\{\zeta<\lambda\mid \otp(\varrho_\alpha^{-1}[\zeta]\cap[\pi_\alpha(\xi),\pi_\alpha(\zeta))\ge\sigma\}$.
\end{itemize}
$f_0$ is well-defined since $\alpha \in \acc^+(\acc^+(X) \cap T)$.
$f_1$ is well-defined since $f[\alpha] \subseteq \alpha$.
$f_2$ is well-defined since $\rng(\pi_\alpha)$ is cofinal in $\alpha$.
$f_3$ is well-defined since
$\cf(\alpha)=\lambda$, so that $\otp([\pi_\alpha(\xi),\alpha))\ge\lambda>\sigma$ for all $\xi<\lambda$.

Define $f_*:\lambda\rightarrow\lambda$ by stipulating:
$$f_*(\xi):=\max\{f_0(\xi),f_1(f_0(\xi)),f_2(f_1(f_0(\xi))),f_3(\xi)\}.$$

From now on, work in $V[G]$.
By the choice of $\Lambda$, pick a large enough $\xi\in\Lambda$ such that $f_*(\xi)<\min(\Lambda\setminus(\xi+1))$ and $\pi_\alpha(\xi)>\tau$. Denote $\zeta:=\min(\Lambda\setminus(\xi+1))$.
Clearly, $C_\alpha\cap[\pi_\alpha(\xi),\pi_\alpha(\zeta))=\overline{x}$, where
$$x:=[\pi_\alpha(\xi),\pi_\alpha(\zeta))\cap (\{\pi_\alpha(\xi)\}\cup \varrho_\alpha^{-1}[\zeta]).$$

Put $\zeta_0:=f_0(\xi)$, $\zeta_1:=f_1(\zeta_0)$, $\zeta_2:=f_2(\zeta_1)$, and $\zeta_3:=f_3(\xi)$.
Evidently, $\zeta_0, \zeta_1 \in \rng(\varrho_\alpha)$.
By $f_*(\xi)<\zeta$, we have $\zeta_i<\zeta$ for all $i<4$.

Write $\beta:=\varrho_\alpha^{-1}(\zeta_0)$, $j:=f(\beta)$ and $\gamma:=\sup(x\cap\beta)$.
Then $\beta\in T$, $j=\varrho_\alpha^{-1}(\zeta_1)$, and $X^j=X\cap\beta$ is a cofinal subset of $\beta$.
As $\zeta_2<\zeta$, $\min(x)=\pi_\alpha(\xi)\notin T$ and $\beta\in T\setminus \pi_\alpha(\xi)$, we infer that
$$\tau<\pi_\alpha(\xi)\le\gamma<\beta\le j\le \pi_\alpha(\zeta_2)<\pi_\alpha(\zeta)<\alpha.$$

By $\zeta_0<\zeta$ and $\zeta_1<\zeta$, we altogether have $\{\beta,j\}\s x\setminus\{\min(x)\}$.
Fix some nonzero $i<\otp(x)$ such that $j=x(i)$. Recalling that $X^j=X\cap\beta$ is a cofinal subset of $\beta$, we have:
$$h_{x,\beta}(i)=\{\epsilon_{x,\beta}(i)\}\cup\suc_{\otp(x)}(X\cap\beta\setminus \epsilon_{x,\beta}(i)),$$
where $\otp(h_{x,\beta}(i)) = 1+\otp(x)$ and $h_{x,\beta}(i) \subseteq (\gamma,\beta)$.

As $f_3(\xi)=\zeta_3<\zeta$, we have $\otp(x)\ge\sigma$, and so in particular
$$\suc_\sigma(\bigcup\rng(h_{x,\beta})\setminus\epsilon_{x,\beta}(i))\s X.$$

As $\beta\in x\cap T$, $\beta\neq\min(x)$, and $\gamma = \sup(x\cap\beta)$, we know that $C_\alpha\cap[\gamma,\beta) = \overline{x} \cap [\gamma,\beta)$ is equal to the ordinal closure of $\bigcup\rng(h_{x,\beta})$.
Put $\epsilon:=\epsilon_{x,\beta}(i)$.
Then $\epsilon\in C_\alpha \cap (\gamma,\beta)\s C_\alpha \setminus \tau$, and $\suc_\sigma(C_\alpha\setminus\epsilon)\s X$, as sought.
\end{proof}

For every $\alpha\in T$, let $C_\alpha^\bullet:=\rng(g_\alpha)$, where
$g_\alpha:C_\alpha\rightarrow\alpha$ is defined by stipulating:\label{def_of_g}
$$g_\alpha(\beta):=
\begin{cases}
\beta,   &\text{if } \beta \in \acc(C_\alpha); \\
\min(Z_\beta\cup\{\beta\}),  &\text{if } \beta = \min(C_\alpha); \\
\min((Z_\beta\cup\{\beta\})\setminus (\sup(C_\alpha \cap \beta)+1)),  &\text{otherwise.}
\end{cases}$$

Note that  $C_\alpha^\bullet(i)\le C_\alpha(i)< C_\alpha^\bullet(i+1)$ for all $i<\otp(C_\alpha)$.
So $g_\alpha$ is strictly increasing and continuous, and $C^\bullet_\alpha$ is a club subset of $\alpha$ with $\otp(C^\bullet_\alpha) = \otp(C_\alpha) \leq \lambda$,
$\acc(C^\bullet_\alpha) = \acc(C_\alpha)$,
and $\nacc(C^\bullet_\alpha) = g_\alpha[\nacc(C_\alpha)]$.

\begin{claim}\label{c335} For every cofinal subset $A\s\kappa$ (from $V[G]$),
there exists a club $E_A\s\kappa$, such that for every $\alpha\in T\cap E_A$ and every $\sigma<\lambda$, we have $\sup\{\eta\in C^\bullet_\alpha\mid \suc_\sigma(C^\bullet_\alpha\setminus\eta)\s A\}=\alpha$.
\end{claim}
\begin{proof} Given a cofinal subset $A\s\kappa$ from $V[G]$,
let $X_A\in\mathcal P^{V}(\kappa)$ be given by Claim~\ref{c743}.
In particular, $V\models X_A\text{ is a stationary subset of }\kappa$.
As $V[G]\models \acc^+(A)\text{ is a club in }\kappa$,
and $V[G]$ is a \textit{$\kappa$-cc} forcing extension of $V$, we may fix some  $B\in V$ such that $V\models B\text{ is a club in }\kappa$,
and $V[G]\models B\s\acc^+(A)$. Put $X:=X_A\cap B$. Then:
\begin{enumerate}
\item $V\models X\text{ is stationary in }\kappa$;
\item $V[G]\models X\s \{\beta<\kappa\mid Z_\beta=A\cap\beta\text{ is a cofinal subset of }\beta\}$.
\end{enumerate}

Now, let $E_A:=D_X$, where $D_{X}$ is given by Claim~\ref{c742}.
Let $\alpha\in T\cap E_A$ and $\sigma<\lambda$ be arbitrary. Then $\sup\{\epsilon\in C_\alpha\mid \suc_\sigma(C_\alpha\setminus\epsilon)\s  X\}=\alpha$.
To see that
$\sup\{\eta\in C^\bullet_\alpha\mid \suc_\sigma(C^\bullet_\alpha\setminus\eta)\s A\}=\alpha$,
let $\tau < \alpha$ be arbitrary.
Choose $\epsilon \in C_\alpha$ such that
$\suc_\sigma(C_\alpha\setminus\epsilon)\s  X$ and $g_\alpha(\epsilon) > \tau$.
Let $\eta := g_\alpha(\epsilon)$.
To see that $\suc_\sigma(C^\bullet_\alpha\setminus\eta)\s A$,
consider an arbitrary $i<\sigma$, and we will show that $(C^\bullet_\alpha\setminus\eta)(i+1) \in A$.

Write $\beta := (C_\alpha\setminus\epsilon)(i+1)$.
Then $\beta \in \nacc(C_\alpha)$, and $\beta \in X$, so that $Z_\beta=A\cap\beta$ is a cofinal subset of $\beta$, and hence
$$(C^\bullet_\alpha \setminus \eta)(i+1)=g_\alpha(\beta)=\min((Z_\beta\cup\{\beta\})\setminus (\sup(C_\alpha \cap \beta)+1))=\min(A\setminus(\sup(C_\alpha \cap \beta)+1))\in A,$$
as required.
\end{proof}

Let $\mathcal C_0^\bullet:=\{\emptyset\}$, and $\mathcal C_{\delta+1}^\bullet:=\{\{\delta\}\}$ for all $\delta<\kappa$.
For every $\delta\in T$, let $\mathcal C_\delta^\bullet:=\{C_\delta^\bullet\}$,
and for every $\delta\in\acc(\kappa)\setminus T$, let
$$\mathcal C_\delta^\bullet:=\{ C_\alpha^\bullet\cap\delta\mid \alpha\in T, \sup(C_\alpha^\bullet\cap\delta)=\delta\}\cup\{ c\in[\delta]^{<\lambda}\cap V\mid c\text{ is a club in }\delta\}.$$

\begin{claim} $|\mathcal C_\delta^\bullet| < \kappa$ for all $\delta<\kappa$.
\end{claim}
\begin{proof} Suppose not, and let $\delta\in\acc(\kappa)\setminus T$ be a counterexample.
As $V\models |[\delta]^{<\lambda}|\le\lambda < \kappa$,
and by Claim~\ref{c8103}, let us pick $\alpha<\alpha'$ both from $T$ and above $\delta$ such that $C_\alpha^\bullet\cap\delta\neq C_{\alpha'}^\bullet\cap\delta$,
$\sup(C_\alpha^\bullet\cap\delta)=\delta=\sup(C_{\alpha'}^\bullet\cap\delta)$, and $C_\alpha\cap\delta=C_{\alpha'}\cap\delta$.
By $\delta\in\acc(C_\alpha^\bullet)$, we have $g_\alpha(\delta)=\delta$, and hence $C_\alpha^\bullet\cap\delta=g_\alpha[\delta]$.
Likewise, $C_{\alpha'}^\bullet\cap\delta=g_{\alpha'}[\delta]$. So $g_\alpha\restriction\delta\neq g_{\alpha'}\restriction\delta$,
contradicting the fact that $C_\alpha\cap\delta=C_{\alpha'}\cap\delta$.
\end{proof}

Altogether, $\langle \mathcal C_\delta^\bullet\mid\delta<\kappa\rangle$ witnesses $\p^*(T,\lambda)$.
\end{proof}

\begin{cor} Suppose that $\lambda$ is a strongly inaccessible cardinal satisfying $2^\lambda=\lambda^+$.

If $\mathbb P$ is a \textit{$\lambda^+$-cc} notion of forcing of size $\le2^\lambda$
that makes $\lambda$ into a singular cardinal,
then $\mathbb P$ introduces a free $\lambda^+$-Souslin tree.
\end{cor}
\begin{proof} By Proposition~\ref{p31}, $\mathbb P\in\mathbb C_\lambda$.
By Theorem~\ref{thm46}, then, in $V^{\mathbb P}$, $\p^*(T,\lambda)+\diamondsuit(\lambda^+)$ holds for some stationary subset $T$ of $\lambda^+$.

Finally, by \cite{rinot23}, $\p^*(T,\lambda)+\diamondsuit(\lambda^+)$ entails $\p(\lambda^+,\lambda^+,{\sq},\lambda^+,\{\lambda^+\},2,1,\mathcal E_{\lambda^+})$,
which by the arguments of \cite[$\S6$]{rinot20} suffices for the construction of a free $\lambda^+$-Souslin tree.
\end{proof}

By Theorem~\ref{thm46}, if $\theta<\lambda=\lambda^{<\lambda}$ are infinite regular cardinals,
and $\ch_\lambda$ holds, then $V^{\col(\theta,\lambda)}\models\p^*(T,\lambda)+\diamondsuit(\kappa)$,
where $\kappa := \lambda^+$ and $T := E^\kappa_\lambda$,
provided a fact we have already mentioned but did not prove: $\col(\theta,\lambda) \in \mathbb C_\lambda$ in this scenario.

The next proposition shows that moreover $\p^*(T,|\lambda|)$ holds in the extension.
A byproduct of its proof, will also establish that indeed $\col(\theta,\lambda) \in \mathbb C_\lambda$.

\begin{prop}\label{p34} Suppose that $\theta<\lambda=\lambda^{<\theta}$ are infinite regular cardinals,
and $\ch_\lambda$ holds.

Let $\kappa:=\lambda^+$ and $T:=E^{\kappa}_\lambda$. Then $V^{\col(\theta,\lambda)}\models\p^*(T,\theta)+\diamondsuit(\kappa)$.
\end{prop}
\begin{proof} Work in $V$. For every nonzero $\alpha<\lambda^+$, fix a surjection $f_\alpha:\lambda\rightarrow\alpha$.
Work in $V[G]$, where $G$ is $\col(\theta,\lambda)$-generic.
Put $g:=\bigcup G$. By genericity,  $g:\theta\rightarrow\lambda$ is a \emph{shuffling surjection}, that is, it satisfies $|g^{-1}\{\eta\}|=\theta$ for every $\eta<\lambda$.

Note that as $\col(\theta,\lambda)$ is $({<\theta})$-closed, for every $\alpha\in T$, we have $\theta\le\cf(\alpha)\le|\alpha|=|\lambda|=\theta$.
That is, $T\s E^{\kappa}_\theta$.
Fix $\alpha\in T$.
Put $g_\alpha:=f_\alpha\circ g$.
As $g$ is a shuffling surjection, so is $g_\alpha : \theta \to \alpha$.  Then, since $\cf(\alpha)=\theta$, we may define a strictly increasing function $h_\alpha:\theta\rightarrow\theta$ by recursion:
\begin{itemize}
\item $h_\alpha(0):=0$;
\item  $h_\alpha(i+1):=\min\{ k\in(h_\alpha(i),\theta)\mid g_\alpha(k)>g_\alpha(h_\alpha(i))\}$;
\item for $i\in\acc(\theta)$, $h_\alpha(i):=\min\{k\in(\sup(\rng(h_\alpha\restriction i)),\theta)\mid g_\alpha(k)=\sup(\rng(g_\alpha\circ(h_\alpha\restriction i)))\}$.
\end{itemize}

Clearly, $C_\alpha:=\rng(g_\alpha\circ h_\alpha)$ is a subset of $\alpha$ of order-type $\theta$ satisfying
$\acc^+(C_\alpha) \subseteq C_\alpha$. It follows from the next claim (taking $X:=\kappa$) that $C_\alpha$ is moreover a club in $\alpha$.

\begin{claim}\label{c341} For every cofinal subset $X\s\kappa$ from $V$, every $\alpha\in T\cap\acc^+(X)$,
and every $\sigma<\theta$,
we have $\sup\{\epsilon\in C_\alpha\mid \suc_\sigma(C_\alpha\setminus\epsilon)\s  X\}=\alpha$.\footnote{Let us point out that $C_\lambda$
witnesses the crucial reason for $\col(\theta,\lambda)\in\mathbb C_\lambda$.
That is, by taking $\Lambda:=C_\lambda$, we get that $\Lambda$ is a club in $\lambda$ of order-type $\cf(\lambda)$ such that
for every function $f\in{}^\lambda\lambda\cap V$, there exists some $\xi\in\Lambda$ with $f(\xi)<\min(\Lambda\setminus(\xi+1))$.
Indeed, given $f\in{}^\lambda\lambda\cap V$, simply pick $X\s\kappa$ from $V$ such that $X\setminus\lambda=\kappa\setminus\lambda$,
and such that $X\cap\lambda$ is a cofinal subset of $\lambda$ satisfying that for all $\xi\in X\cap\lambda$, $f(\xi)<\min(X\setminus(\xi+1))$. Now,
appeal to our claim with this $X$ and with $\sigma=3$.}
\end{claim}
\begin{proof} Work in $V$. Let $X,\alpha$ and $\sigma$ be as in the hypothesis.
Let $\varepsilon<\alpha$ be arbitrary.

To run a density argument, let us fix an arbitrary condition $p\in\col(\theta,\lambda)$.
By extending $p$, we may assume that $\bar\theta:=\dom(p)$ is a nonzero ordinal.
Put $p_\alpha:=f_\alpha\circ p$, and let $\chi$ be the largest ordinal
for which there exists a strictly increasing function $h:\chi\rightarrow\bar\theta$ satisfying:
\begin{itemize}
\item $h(0)=0$;
\item  $h(i+1)=\min\{ k\in(h(i),\bar\theta)\mid p_\alpha(k)>p_\alpha(h(i))\}$;
\item for $i\in\acc(\chi)$, $h(i)=\min\{k\in(\sup(\rng(h\restriction i)),\bar\theta)\mid p_\alpha(k)=\sup(\rng(p_\alpha\circ(h\restriction i)))\}$.
\end{itemize}

Let $\beta:=\sup(\rng(p_\alpha\circ h))$, and $\epsilon:=\max\{\beta,\varepsilon\}$. As $\left|\rng(p_\alpha)\right|\le|\bar\theta|<\theta=\cf(\alpha)$, we know that $\epsilon<\alpha$.
For all $\iota<\theta$, denote $x_\iota:=\{\beta,\varepsilon\}\cup\suc_\iota(X\setminus\epsilon)$.

Pick a large enough $\iota<\theta$ such that $\otp(\suc_\sigma(x_\iota\setminus\epsilon))=\sigma$,
and let $\overline{x}$ denote the ordinal closure of $x_\iota$.
By $\iota<\theta$ and $\alpha\in\acc^+(X)\cap T$,  we have $\overline{x}\s \rng(f_\alpha)$.
Pick a condition $q$ extending $p$, such that $\langle q(k)\mid \bar\theta\le k<\bar\theta+\otp(\overline{x})\rangle$ forms the increasing enumeration of $\overline{x}$.
Then $q$ forces that $\epsilon\in C_\alpha\setminus\varepsilon$ and $\suc_\sigma(C_\alpha\setminus\epsilon)\s X$.
\end{proof}

In $V$, by $\ch_\lambda$ and the main result of \cite{Sh:922}, $\diamondsuit(\kappa)$ holds.
By $\lambda^{<\theta}=\lambda<\kappa$, $\col(\theta,\lambda)$ is a \textit{$\kappa$-cc} poset of size $<\kappa$,
so that the same argument from the proof of Theorem~\ref{thm46} yields that in $V[G]$, there exists a sequence $\langle Z_\alpha\mid\alpha<\kappa\rangle$
satisfying that for every $A\in\mathcal P^{V[G]}(\kappa)$, there exists some $X_A\in\mathcal P^{V}(\kappa)$ such that:
\begin{enumerate}
\item $V\models X_A\text{ is stationary}$;
\item $V[G]\models X_A\s \{\beta<\lambda^+\mid Z_\beta=A\cap\beta\}$.
\end{enumerate}

As $(\theta^+)^{V[G]}=(\lambda^+)^V=\kappa$,
and as $\col(\theta,\lambda)$ is a \textit{$\kappa$-cc} poset, stationary subsets of $\kappa$ from $V$ remain stationary in $V[G]$,
so that $V[G]\models\diamondsuit(\kappa)$.

Work in $V[G]$.
For every $\alpha\in T$, let $C_\alpha^\bullet:=\rng(g_\alpha)$, where
$g_\alpha:C_\alpha\rightarrow\alpha$ is defined exactly as in the proof of Theorem~\ref{thm46} (page~\pageref{def_of_g}).
Then $C_\alpha^\bullet$ is again a club in $\alpha$ of order-type $\theta$.
\begin{claim} For every cofinal subset $A\s\kappa$ (from $V[G]$),
there exists a club $E_A\s\kappa$, such that for every $\alpha\in T\cap E_A$ and every $\sigma<\theta$, we have $\sup\{\eta\in C^\bullet_\alpha\mid \suc_\sigma(C^\bullet_\alpha\setminus\eta)\s A\}=\alpha$.
\end{claim}
\begin{proof} The result follows from Claim~\ref{c341} just like Claim~\ref{c335} follows from Claim~\ref{c742}.
\end{proof}

For every $\alpha\in E^{\kappa}_\theta\setminus T$, let $C_\alpha^\bullet$ be an arbitrary club subset of $\alpha$ of order-type $\theta$.

Finally, for every $\alpha<\kappa$, let:
$$\mathcal C_\alpha^\bullet:=\begin{cases}\{\emptyset\},&\text{if }\alpha=0;\\
\{\{\beta\}\},&\text{if }\alpha=\beta+1;\\
\{C_\alpha^\bullet\},&\text{if }\alpha\in E^{\kappa}_\theta;\\
\{ c\in[\alpha]^{<\theta}\mid c\text{ is a club in }\alpha\},&\text{otherwise}.\end{cases}$$

As $V\models \lambda^{<\theta}=\lambda$, and $\col(\theta,\lambda)$ is $({<\theta})$-closed, we have $|[\lambda]^{<\theta}|\le|\lambda|=\theta$,
so that $|\mathcal C_\alpha^\bullet|<\kappa$ for all $\alpha\in E^{\kappa}_{<\theta}$.
Altogether, $\langle \mathcal C_\alpha^\bullet\mid\alpha<\kappa\rangle$ witnesses $\p^*(T,\theta)$.
\end{proof}

\begin{prop} Suppose that $\lambda^{<\lambda}=\lambda$ is an infinite cardinal,
and $\kappa>\lambda$ is a strongly inaccessible cardinal.
If $\mathbb P$ is a $({<\lambda})$-distributive, $\kappa$-cc notion of forcing,
collapsing $\kappa$ to $\lambda^+$,
then $V^{\mathbb P}\models\p^*(E^{\lambda^+}_\lambda,\lambda)+\diamondsuit(\lambda^+)$.
\end{prop}
\begin{proof} As $\kappa$ is strongly inaccessible,
for every $\alpha<\kappa$, the collection $\mathcal N_\alpha:=\{ \tau\in V_{\alpha+1}\mid \tau\text{ is a }\mathbb P\text{-name}\}$ has size $<\kappa$.
Let $G$ be $\mathbb P$-generic over $V$, and work in $V[G]$. For every $\mathbb P$-name $\tau$, denote by $\tau_G$ its interpretation by $G$.
Then, for all $\alpha<\kappa$, let $\mathcal A_\alpha:=\{ \tau_G\mid \tau\in\mathcal N_\alpha\}\cap\mathcal P(\alpha)$.
Since $\mathbb P$ has the \textit{$\kappa$-cc}, $\langle \mathcal A_\alpha\mid \alpha<\kappa\rangle$ forms a $\diamondsuit^+(\lambda^+)$-sequence,
and in particular, $\diamondsuit^*(E^{\lambda^+}_{\lambda})$ and $\diamondsuit(\lambda^+)$ hold.
Since $\mathbb P$ is $({<\lambda})$-distributive, we still have $\lambda^{<\lambda}=\lambda$.
Finally, by \cite{rinot23}, $\diamondsuit^*(E^{\lambda^+}_{\lambda})+\lambda^{<\lambda}=\lambda$ entails $\p^*(E^{\lambda^+}_\lambda,\lambda)$.
\end{proof}

\end{document}